\documentclass[11pt]{amsart}

\usepackage[margin=1.1in]{geometry}
\usepackage{amsmath,amssymb,amsfonts,amsthm,mathtools}
\usepackage{mathrsfs}
\usepackage{enumitem}
\usepackage[colorlinks=true,citecolor=blue,linkcolor=blue,urlcolor=blue]{hyperref}
\usepackage{tikz}
\usetikzlibrary{arrows.meta,calc,shadings}
\hypersetup{
  colorlinks=true,
  linkcolor=blue,
  citecolor=blue,
  urlcolor=blue
}

\newtheorem{theorem}{Theorem}[section]
\newtheorem{proposition}[theorem]{Proposition}
\newtheorem{question}{Question}[section]
\newtheorem{lemma}[theorem]{Lemma}
\newtheorem{corollary}[theorem]{Corollary}

\numberwithin{equation}{section}

\DeclareMathOperator{\Spec}{Spec}
\DeclareMathOperator{\diam}{diam}
\DeclareMathOperator{\dist}{dist}
\DeclareMathOperator{\Tr}{Tr}

\newcommand{\R}{\mathbb R}
\newcommand{\HH}{\mathcal H}
\newcommand{\EE}{\mathcal E}
\newcommand{\Om}{\Omega}
\newcommand{\eps}{\varepsilon}
\newcommand{\mm}{\mathfrak m}
\newcommand{\RCD}{\mathrm{RCD}}

\title[Non-degenerate Local Maxima on Positively Curved two-Spheres]{Arbitrarily Many Non-degenerate Local Maxima of First Nonzero Eigenfunctions on Positively Curved Two-spheres}
\author{Heng Zhang}
\address{School of Mathematical Sciences, University of Science and Technology of China, Hefei, China}
\email{hengz@mail.ustc.edu.cn}
\date{}

\begin{document}

\begin{abstract}
We prove that, for every integer $m\ge 2$, there is a smooth closed Riemannian
surface $(M,g)$ diffeomorphic to $\mathbb S^2$, with positive Gaussian curvature, such
that $\lambda_1(M)$ is simple and, after choosing the sign, its normalized first
nonzero Laplace eigenfunction has at least $m$ distinct non-degenerate local
maxima.  This gives a negative answer to the Open Question of Grossi and
Provenzano (Math. Ann. 389(4): 3447--3470, 2024).
\end{abstract}

\maketitle

\section{Introduction}

\subsection{Background and main result}

The number and configuration of critical points of Laplace eigenfunctions is a
classical problem in spectral geometry.  In dimension two, the local structure of
nodal sets of eigenfunctions goes back to Cheng \cite{Cheng1976}, while critical
sets of solutions of elliptic equations have been studied, for example, by Hardt,
Hoffmann-Ostenhof, Hoffmann-Ostenhof and Nadirashvili
\cite{HardtHoffmannOstenhofNadirashvili1999}.  Jakobson and Nadirashvili
constructed eigenfunctions with few critical points
\cite{JakobsonNadirashvili1999}, and Freitas constructed surfaces whose first
nonzero eigenfunctions have closed nodal lines and interior hot spots
\cite{Freitas2002}.

We use the sign convention that $-\Delta$ is the non-negative Laplacian.  For a
closed connected Riemannian manifold, or more generally for a compact connected
metric-measure space on which the Laplacian has discrete spectrum, we write its
eigenvalues as
\begin{equation}\label{eq:eigenvalue-indexing}
0=\lambda_0<\lambda_1\le \lambda_2\le\cdots,
\end{equation}
counted with multiplicity.  Throughout the paper, a \emph{second eigenfunction}
means an eigenfunction corresponding to $\lambda_1$, the first positive
eigenvalue in the indexing convention \eqref{eq:eigenvalue-indexing}.

Without strong geometric restrictions, the critical or level-set geometry of
eigenfunctions may be very complicated.  For instance, Enciso and Peralta-Salas
proved a prescribed nodal-set realization result for eigenfunctions on compact
manifolds \cite{EncisoPeraltaSalas}.  Buhovsky, Logunov and Sodin constructed a
metric on $\mathbb T^2$ such that, for infinitely many eigenvalues, a
corresponding eigenfunction has infinitely many isolated critical points
\cite{BuhovskyLogunovSodin}.  B\'erard, Charron and Helffer constructed smooth
metrics on $\mathbb T^2$ and $\mathbb S^2$, arbitrarily close to the flat and
round metrics respectively, for which suitable eigenfunctions have infinitely
many connected components of a level-set complement
\cite{BerardCharronHelffer}.  These results show that large nodal, critical, or
level-set complexity is possible for general metrics.  They do not, however,
settle what happens at the first positive eigenvalue under a positive Gaussian
curvature assumption.

This leads naturally to the question of whether curvature or convexity-type
hypotheses can restore a rigid critical-point structure for low-energy
eigenfunctions.  In the planar theory, convexity is known to impose strong
restrictions in several related problems.  Classical log-concavity and convexity
results for first Dirichlet eigenfunctions and associated variational problems
go back to Brascamp--Lieb \cite{BrascampLieb1976} and Caffarelli--Spruck
\cite{CaffarelliSpruck1982}.  For semilinear elliptic equations,
Cabr\'e--Chanillo \cite{CabreChanillo1998} and
De Regibus--Grossi--Mukherjee \cite{DeRegibusGrossiMukherjee2021} obtained
representative uniqueness results for critical points in convex planar domains.
For second Dirichlet eigenfunctions in convex planar domains, the number of
critical points was studied by De Regibus and Grossi
\cite{DeRegibusGrossi2022}.  Thus, from the point of view of convexity and
curvature, it is natural to ask whether analogous uniqueness phenomena persist
for second eigenfunctions on closed positively curved surfaces.

Grossi and Provenzano \cite{GrossiProvenzanoPublished} recently pursued this
theme in two directions.  They obtained uniqueness results for critical points of
semi-stable solutions on convex domains in two-dimensional model spaces, and
they also studied second eigenfunctions on closed manifolds of revolution.  For a
closed surface of revolution diffeomorphic to $\mathbb S^2$, their curvature
condition is equivalent to positive Gaussian curvature.  Motivated by this
setting, they asked the following question
\cite[Open Question~3.19]{GrossiProvenzanoPublished}.

\begin{question}\label{question}
Let $M$ be a closed Riemannian surface diffeomorphic to $\mathbb S^2$ of positive
Gaussian curvature, and let $u$ be a second eigenfunction of the Laplacian on
$M$.  Is it true that $u$ has two non-degenerate critical points, a maximum and a
minimum?
\end{question}

\begin{figure}[htbp]
\centering
\begin{tikzpicture}[>=Latex, every node/.style={font=\small}]

  \begin{scope}[shift={(0,0)}]
    \shade[ball color=gray!25] (0,0) circle (1.5);
    \draw[thick] (0,0) circle (1.5);

    \draw[blue!60!black, densely dotted] (-1.05,0.72) arc (180:360:1.05 and 0.17);
    \draw[blue!60!black]                 ( 1.05,0.72) arc (  0:180:1.05 and 0.17);

    \draw[blue!60!black, dashed]         (-1.50,0.00) arc (180:360:1.50 and 0.28);
    \draw[blue!60!black]                 ( 1.50,0.00) arc (  0:180:1.50 and 0.28);

    \draw[blue!60!black, densely dotted] (-1.05,-0.72) arc (180:360:1.05 and 0.17);
    \draw[blue!60!black]                 ( 1.05,-0.72) arc (  0:180:1.05 and 0.17);

    \fill[red!80!black] (0,1.5) circle (2.2pt);
    \node[red!80!black, align=center] at (0,2.05)
      {maximum};

    \fill[blue!80!black] (0,-1.5) circle (2.2pt);
    \node[blue!80!black, align=center] at (0,-2.10)
      {minimum};

    \draw[densely dotted] (0,1.5) -- (2.1,1.5);
    \draw[densely dotted] (0,-1.5) -- (2.1,-1.5);
    \draw[->,thick] (2.1,-1.55) -- (2.1,1.75);
    \node at (2.05,1.95) {$h=z$};

    \node[align=center, text width=4.8cm] at (0,-3.05)
      {(a) $\mathbb S^2$ with the round metric:\\
       $h|_{\mathbb S^2}$ has exactly\\
       two critical points.};
  \end{scope}

  \begin{scope}[shift={(7.4,0.05)}]
    \shade[top color=gray!8,bottom color=gray!20]
      (-2.5,0)
      .. controls (-1.8,0.72) and (1.8,0.72) .. (2.5,0)
      .. controls (1.8,-0.72) and (-1.8,-0.72) .. cycle;

    \draw[thick]
      (-2.5,0)
      .. controls (-1.8,0.72) and (1.8,0.72) .. (2.5,0)
      .. controls (1.8,-0.72) and (-1.8,-0.72) .. cycle;

    \draw[thick, gray!70!black]
      (-2.1,0.15) .. controls (-0.8,0.42) and (0.8,0.42) .. (2.1,0.15);

    \draw[dashed, gray!65!black]
      (-2.0,-0.12) .. controls (-0.7,-0.42) and (0.7,-0.42) .. (2.0,-0.12);

    \foreach \x/\lab in {-1.85/p_1, -0.65/p_2, 0.65/p_3, 1.85/p_m}{
      \fill[red!80!black] (\x,0.20) circle (2.2pt);
      \draw[red!80!black, -{Latex[length=2mm]}] (\x,0.22) -- ++(0,0.28);
      \node[red!80!black, above] at (\x,0.45) {$\lab$};
    }

    \node[red!80!black, align=center] at (0,1.35)
      {several non-degenerate\\local maxima of $u$};

    \node[gray!70!black] at (0,-1.45)
      {seam of a doubled convex domain};

    \node[align=center, text width=5.5cm] at (0,-3.20)
      {(b) Positive curvature does not force\\
       the two-critical-point picture.};
  \end{scope}

\end{tikzpicture}
\caption{A schematic contrast related to Question~\ref{question}. 
On the round sphere, the height function has exactly two critical points,
a maximum and a minimum. In contrast, on suitably chosen positively curved spheres,
a second eigenfunction may have several distinct non-degenerate local maxima.}
\label{fig:question-motivation}
\end{figure}
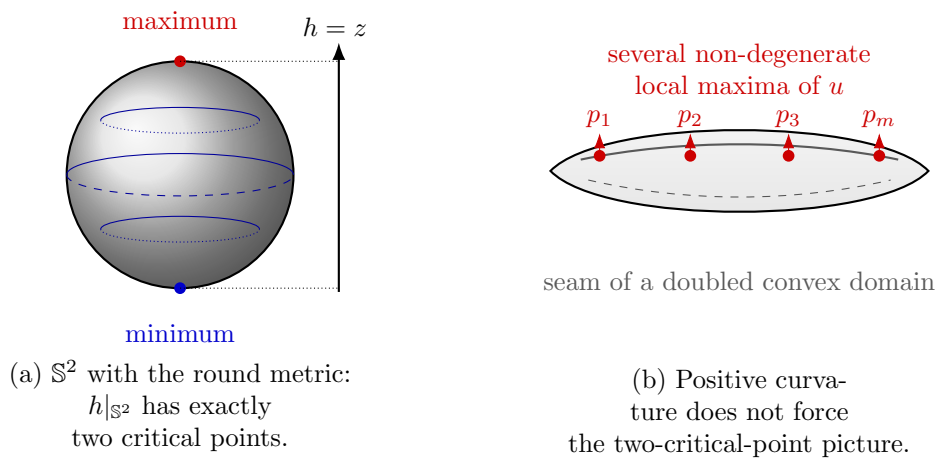

The main result of this paper gives a negative answer in a strong form.  Positive
Gaussian curvature alone does not force a second eigenfunction on a two-sphere to
have only one maximum and one minimum; in fact, the number of non-degenerate
local maxima can be made arbitrarily large.

\begin{theorem}\label{thm:main}
For every integer $m\ge 2$ there exists a smooth embedded surface
$M_m\subset\R^3$, diffeomorphic to $\mathbb S^2$, whose induced Riemannian metric
has positive Gaussian curvature, such that $\lambda_1(M_m)$ is simple and the
normalized $\lambda_1(M_m)$-eigenfunction, after choosing its sign, has at least
$m$ distinct non-degenerate local maxima.
\end{theorem}

Taking $m=2$ in Theorem~\ref{thm:main}, one obtains a positively curved
Riemannian two-sphere whose second eigenfunction, with a suitable sign choice,
has at least two distinct non-degenerate local maxima.  Since every local
maximum of a smooth eigenfunction is a critical point, such an eigenfunction
cannot have exactly two critical points consisting of one maximum and one
minimum.  Thus Question~\ref{question} has a negative answer.

\subsection{Strategy for the construction}

The proof is based on a degeneration to a doubled planar domain.  The key point
is to find a planar first eigenfunction whose extrema are already sufficiently
non-rigid, and then to transfer this behavior to a closed positively curved
sphere.

For a bounded connected planar domain $U$, we write
\[
0=\mu_0^N(U)<\mu_1^N(U)\le\mu_2^N(U)\le\cdots
\]
for the Neumann spectrum and
\[
0<\lambda_1^D(U)\le\lambda_2^D(U)\le\cdots
\]
for the Dirichlet spectrum, in both cases counted with multiplicity.

The relevant planar tool comes from the classical hot spots problem for Neumann
eigenfunctions.  Rauch's hot spots conjecture predicts that the extrema of a
first nontrivial Neumann eigenfunction of a bounded insulated domain should occur
on the boundary \cite{Rauch1975}.  Although the conjecture is false for general
planar domains \cite{BurdzyWerner1999,BanuelosBurdzy1999}, many positive results
are known under additional geometric assumptions; see, for instance,
\cite{JerisonNadirashvili2000,AtarBurdzy2004,Pascu2002,Steinerberger2020,JudgeMondal2020,ChenGuiYao2026}.

Miyamoto's construction \cite{MiyamotoJJIAM} exhibits a different phenomenon,
which is precisely the one needed here.  It does not provide a failure of the
boundary hot spots conclusion in convex domains.  Rather, it shows that even in
the class of convex polygons, the boundary hot spots of a first positive Neumann
eigenfunction need not consist of a single maximum and a single minimum.  More
precisely, in \cite[Theorem A]{MiyamotoJJIAM}, one obtains a
convex polygon $\Om$ whose first positive Neumann eigenvalue is simple and whose
corresponding Neumann eigenfunction $\phi$ has $m$ isolated boundary maxima.

We then pass from this planar Neumann problem to a closed surface by forming the
metric double of $\Om$,
\[
X=D\Om,
\]
obtained by gluing two copies of $\overline\Om$ along $\partial\Om$.  This double
is an Alexandrov sphere of curvature at least zero.  The involution exchanging
the two sheets decomposes $L^2(X)$, and hence the spectrum of $X$, into even and
odd parts.  The even part is unitarily equivalent to the Neumann spectrum of
$\Om$, while the odd part is unitarily equivalent to the Dirichlet spectrum of
$\Om$.  By Filonov's strict Friedlander inequality \cite{Filonov}, we have
\[
\lambda_1(X)
=
\min\{\mu_1^N(\Om),\lambda_1^D(\Om)\}
=
\mu_1^N(\Om),
\]
so the first positive eigenvalue of the double is realized in the even part of
the spectral decomposition.  Since $\mu_1^N(\Om)$ is simple, the
$\lambda_1(X)$-eigenspace is spanned by the even double $\Phi$ of Miyamoto's
Neumann eigenfunction.  After choosing the sign of $\phi$ so that Miyamoto's hot
spots are maxima, these boundary hot spots become strict local maxima of $\Phi$
at the corresponding seam points of the double.

The remaining task is to replace the singular Alexandrov sphere $X$ by smooth
positively curved spheres while preserving both the first eigenfunction and its
local maxima.  We do this by a pancake approximation.  First we replace the
polygon by smooth strictly convex inner domains
\[
\Om_\delta=\{\psi_\delta\le 1\},
\qquad
\psi_\delta(x)=\sum_{\alpha=1}^{L} \exp(\ell_\alpha(x)/\delta),
\]
where the affine functions $\ell_\alpha$ define the sides of $\Om$, equivalently
\[
\Om=\bigcap_{\alpha=1}^{L}\{\ell_\alpha\le 0\}.
\]
Over each $\Om_\delta$ we take the thin strictly convex body
\[
C_{\eps,\delta}
=
\left\{(x,z)\in\R^2\times\R:
\psi_\delta(x)+\frac{z^2}{\eps^2}\le 1
\right\},
\qquad
M_{\eps,\delta}:=\partial C_{\eps,\delta}.
\]
The surfaces $M_{\eps,\delta}$ are smooth embedded two-spheres with positive
Gaussian curvature.  For fixed $\delta$, as $\eps\downarrow0$, the upper and
lower graphs
\[
z=\pm \eps\sqrt{1-\psi_\delta(x)}
\]
collapse to the metric double $D\Om_\delta$; then $D\Om_\delta$ converges to
$D\Om$ as $\delta\downarrow0$.  Choosing a diagonal sequence gives smooth
positively curved pancakes $M_j\subset\R^3$ converging to $X$ in measured
Gromov--Hausdorff sense.

Next, the spectral and topological features of the limit survive along this
approximation.  The spaces involved lie in the compact $\RCD(0,2)$ setting, and
$\lambda_1(X)$ is simple.  Honda's spectral convergence theorem \cite{Honda}
therefore gives uniform convergence, after choosing signs, of the
$L^2$-normalized $\lambda_1(M_j)$-eigenfunctions to the even doubled eigenfunction
$\Phi$, along suitable measured Gromov--Hausdorff approximation maps.  A
topological persistence lemma then ensures that the strict local maxima of
$\Phi$ on the limit produce distinct local maxima of the corresponding first
eigenfunctions on all sufficiently close smooth pancakes.  More precisely, the
argument gives pairwise disjoint regions on a smooth pancake on which the first
eigenfunction has a strict boundary gap.

The final non-degeneracy step is a generic perturbation argument.  Since the
first positive eigenvalue on the selected pancake is simple, the corresponding
normalized eigenfunction varies continuously under small smooth perturbations of
the metric.  The boundary gaps above therefore persist.  By Uhlenbeck's generic
Morse property for eigenfunctions \cite{Uhlenbeck1976}, one may choose an
arbitrarily small smooth perturbation of the metric for which the first
eigenfunction is Morse.  The positive curvature condition is open in the
$C^2$ topology, so this perturbation can be made within the class of positively
curved metrics.  Finally, the smooth solution of the Weyl isometric embedding
problem for positively curved metrics on $\mathbb S^2$ realizes the perturbed
metric as a smooth strictly convex embedded surface in $\R^3$
\cite{Nirenberg1953,Pogorelov1952}.

\subsection{Paper Organization}

The paper is organized as follows.  Section~\ref{sec:neumann-input} records the
planar Neumann tool.  Section~\ref{sec:double-spectrum} identifies the first
positive spectrum of the metric double.  Section~\ref{sec:pancakes} constructs
smooth positively curved pancake approximations.  Section~\ref{sec:spectral}
combines $\RCD$ spectral convergence with a topological persistence lemma for
strict local maxima and records the generic perturbation lemma used to obtain
non-degeneracy.  The proof of Theorem~\ref{thm:main} is completed in
Section~\ref{sec:proof-main}.

\subsection*{Acknowledgment}
The author thanks Jiangcheng You for helpful discussions.

\section{The planar Neumann tool}\label{sec:neumann-input}

   We use the following form of Miyamoto's polygonal hot-spot construction; see
\cite[Theorem A]{MiyamotoJJIAM} and the construction there.

\begin{theorem}[Miyamoto hot-spot polygons]\label{thm:miyamoto}
For every integer $m\ge 2$ there exists a bounded convex polygon
$\Om\subset\R^2$ such that the first positive Neumann eigenvalue $\mu_1^N(\Om)$
is simple and such that a real $\mu_1^N(\Om)$-eigenfunction $\phi$ has exactly
$m$ global maximum points in $\overline\Om$.  These points all lie on
$\partial\Om$ and are isolated strict local maxima relative to $\overline\Om$.
\end{theorem}


Fix $m\ge2$, and choose $\Om$ and $\phi$ as in
Theorem~\ref{thm:miyamoto}.  Since $\phi$ satisfies the weak Neumann eigenvalue
equation, it also solves the resolvent problem
$$
(-\Delta+1)\phi=(\mu_1^N(\Om)+1)\phi
$$
with homogeneous Neumann boundary condition.  As $\Om$ is a bounded convex
polygon, Grisvard's Neumann regularity theorem on bounded convex domains
\cite[Theorem~3.2.1.3]{Grisvard} gives $\phi\in H^2(\Om)$.  Since $\Om$ is a
bounded Lipschitz domain in $\mathbb R^2$, the limiting Sobolev--Morrey embedding
\cite[Section 1.4]{Grisvard} implies
$$
H^2(\Om)\hookrightarrow C^{0,\alpha}(\overline\Om)
\qquad\text{for every }0\le\alpha<1.
$$
Hence $\phi$ has a continuous representative on $\overline\Om$, and throughout
the sequel $\phi$ denotes this representative. We normalize $\phi$ by
\begin{equation}\label{eq:phi-normalization}
\int_\Om \phi\,dx=0,
\qquad
\int_\Om \phi^2\,dx=1.
\end{equation}
Write
$$
p_1,\ldots,p_m\in\partial\Om
$$
for the isolated boundary maximum points of $\phi$.  Let
\begin{equation}\label{eq:max-set}
\mathcal M_\phi:=\{x\in\overline\Om:\phi(x)=\max_{\overline\Om}\phi\}
\end{equation}
be the compact global maximum set.  By Theorem~\ref{thm:miyamoto},
$\mathcal M_\phi=\{p_1,\ldots,p_m\}$.

\begin{lemma}\label{lem:strict-neighborhoods-domain}
For each $i=1,\ldots,m$ there are pairwise disjoint relatively open
neighborhoods $V_i\subset \overline\Om$ of $p_i$ and a number $\eta>0$ such that
\begin{equation}\label{eq:Vi-meets-max-set}
V_i\cap \mathcal M_\phi=\{p_i\}
\end{equation}
and
\begin{equation}\label{eq:boundary-gap-domain}
\phi(p_i)\ge \sup_{\partial_{\overline\Om} V_i}\phi+4\eta.
\end{equation}
Here $\partial_{\overline\Om}V_i$ denotes the boundary of $V_i$ relative to
$\overline\Om$.
\end{lemma}

\begin{proof}
Since $p_i$ is an isolated point of the compact maximum set \eqref{eq:max-set},
choose pairwise disjoint relative neighborhoods $V_i$ whose closures meet
$\mathcal M_\phi$ only at $p_i$ and whose relative boundaries do not meet
$\mathcal M_\phi$.  This gives \eqref{eq:Vi-meets-max-set}.  The continuous
function $\phi$ is strictly smaller than $\phi(p_i)$ on each compact set
$\partial_{\overline\Om}V_i$.  Taking the minimum of the resulting positive gaps
and dividing by $4$ gives a common $\eta>0$, which proves
\eqref{eq:boundary-gap-domain}.
\end{proof}

\section{The metric double and its first positive spectrum}\label{sec:double-spectrum}

Let
\begin{equation}\label{eq:metric-double}
X=D\Om
\end{equation}
be the metric double of $\Om$: two copies $\overline\Om^+$ and
$\overline\Om^-$ are glued by the identity map on $\partial\Om$.  We equip $X$
with its intrinsic length metric and with the measure $\mm_X:=\HH^2_X$.  More
generally, whenever a metric space $Y$ below is equipped with its two-dimensional
Hausdorff measure, we write $\mm_Y:=\HH^2_Y$.  The space $X$ in
\eqref{eq:metric-double} is a compact polyhedral Alexandrov surface homeomorphic
to $\mathbb S^2$.

For $x\in\overline\Om$, we write $x^\pm$ for the corresponding points in the
two sheets; if $x\in\partial\Om$, the points $x^+$ and $x^-$ are identified in
$D\Om$.  We identify $\partial\Om$ with this seam.  If
$V\subset\overline\Om$ is relatively open, we write
$$
DV:=V^+\cup V^-\subset D\Om,
$$
with boundary points identified.

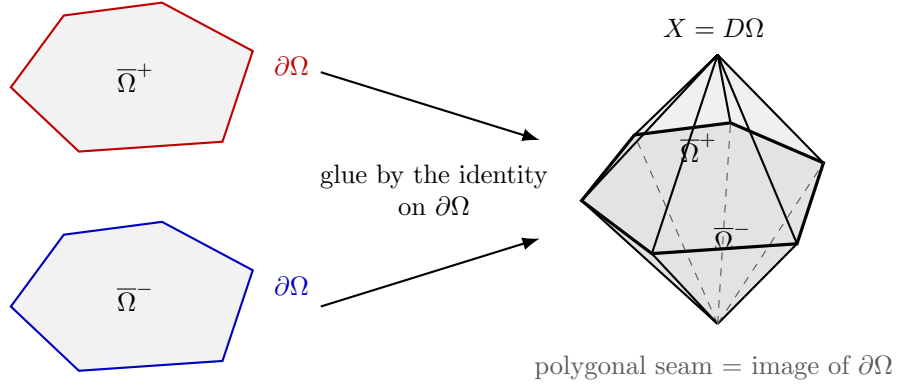
\begin{figure}[htbp]
\centering
\begin{tikzpicture}[>=Latex, every node/.style={font=\small}]

\begin{scope}[shift={(0,1.45)}]
  \coordinate (A+) at (-1.65,-0.10);
  \coordinate (B+) at (-0.95, 0.85);
  \coordinate (C+) at ( 0.35, 1.02);
  \coordinate (D+) at ( 1.55, 0.38);
  \coordinate (E+) at ( 1.15,-0.82);
  \coordinate (F+) at (-0.75,-0.95);

  \fill[gray!10] (A+)--(B+)--(C+)--(D+)--(E+)--(F+)--cycle;
  \draw[thick, red!75!black]
    (A+)--(B+)--(C+)--(D+)--(E+)--(F+)--cycle;

  \node at (0,0.03) {$\overline{\Omega}^{+}$};
  \node[red!75!black, right] at (1.70,0.18) {$\partial\Omega$};
\end{scope}

\begin{scope}[shift={(0,-1.45)}]
  \coordinate (A-) at (-1.65,-0.10);
  \coordinate (B-) at (-0.95, 0.85);
  \coordinate (C-) at ( 0.35, 1.02);
  \coordinate (D-) at ( 1.55, 0.38);
  \coordinate (E-) at ( 1.15,-0.82);
  \coordinate (F-) at (-0.75,-0.95);

  \fill[gray!10] (A-)--(B-)--(C-)--(D-)--(E-)--(F-)--cycle;
  \draw[thick, blue!75!black]
    (A-)--(B-)--(C-)--(D-)--(E-)--(F-)--cycle;

  \node at (0,0.03) {$\overline{\Omega}^{-}$};
  \node[blue!75!black, right] at (1.70,0.18) {$\partial\Omega$};
\end{scope}

\node[gray!70!black, align=center] at (0,2.80)
  {two copies of the same convex polygon};

\draw[->, thick] (2.45, 1.55) -- (5.35, 0.65);
\draw[->, thick] (2.45,-1.55) -- (5.35,-0.65);

\node[align=center] at (3.95,0)
  {glue by the identity\\on $\partial\Omega$};

\begin{scope}[shift={(7.65,0)}]
  \coordinate (A) at (-1.75,-0.15);
  \coordinate (B) at (-1.05, 0.72);
  \coordinate (C) at ( 0.22, 0.88);
  \coordinate (D) at ( 1.45, 0.35);
  \coordinate (E) at ( 1.10,-0.72);
  \coordinate (F) at (-0.82,-0.85);

  \coordinate (T) at (0.05, 1.78);
  \coordinate (S) at (0.05,-1.78);

  \fill[gray!12] (T)--(A)--(B)--cycle;
  \fill[gray!12] (T)--(B)--(C)--cycle;
  \fill[gray!12] (T)--(C)--(D)--cycle;
  \fill[gray!12] (T)--(D)--(E)--cycle;
  \fill[gray!12] (T)--(E)--(F)--cycle;
  \fill[gray!12] (T)--(F)--(A)--cycle;

  \fill[gray!22] (S)--(A)--(B)--cycle;
  \fill[gray!22] (S)--(B)--(C)--cycle;
  \fill[gray!22] (S)--(C)--(D)--cycle;
  \fill[gray!22] (S)--(D)--(E)--cycle;
  \fill[gray!22] (S)--(E)--(F)--cycle;
  \fill[gray!22] (S)--(F)--(A)--cycle;

  \draw[very thick] (A)--(B)--(C)--(D)--(E)--(F)--cycle;

  \draw[thick] (T)--(A) (T)--(B) (T)--(C) (T)--(D) (T)--(E) (T)--(F);
  \draw[thick] (S)--(A) (S)--(F) (S)--(E);
  \draw[dashed, gray!65!black] (S)--(B) (S)--(C) (S)--(D);

  \node at (0,2.13) {$X=D\Omega$};
  \node at (-0.20,0.55) {$\overline{\Omega}^{+}$};
  \node at (0.25,-0.62) {$\overline{\Omega}^{-}$};
  \node[gray!70!black, align=center] at (0,-2.35)
    {polygonal seam = image of $\partial\Omega$};
\end{scope}

\end{tikzpicture}

\caption{The metric double $X=D\Omega$ of a convex polygon $\Omega$.
Two copies $\overline{\Omega}^{+}$ and $\overline{\Omega}^{-}$ are glued by the
identity map along their common boundary $\partial\Omega$.}
\label{fig:metric-double-polygon}
\end{figure}
We first identify the Sobolev space and energy on the double.  We use the
Riemannian convention that the bilinear energy form associated with the Laplacian
is
$$
\EE(f,g)=\int \langle \nabla f,\nabla g\rangle,
$$
and we write the Cheeger energy $\operatorname{Ch}_X$ on $X$ as
$\operatorname{Ch}_X(f)=\frac12\EE_X(f,f)$.

\begin{lemma}[Sobolev structure of the double]\label{lem:sobolev-double}
Let $\Om\subset\R^2$ be a bounded convex polygonal domain and let $X=D\Om$.
Then, identifying a function on $X$ with its restrictions to the two sheets,
\begin{equation}\label{eq:sobolev-double-space}
H^{1,2}(X)=
\left\{(f_+,f_-)\in H^1(\Om)\oplus H^1(\Om):
\Tr f_+=\Tr f_-\text{ on }\partial\Om\right\}.
\end{equation}
Moreover, for $F=(f_+,f_-)$ and $G=(g_+,g_-)$ in $H^{1,2}(X)$,
\begin{equation}\label{eq:double-energy-form}
\EE_X(F,G)=
\int_\Om \nabla f_+\cdot \nabla g_+\,dx+
\int_\Om \nabla f_-\cdot \nabla g_-\,dx .
\end{equation}
\end{lemma}

\begin{proof}
Let $\mathcal V$ denote the trace-matching subspace displayed in
\eqref{eq:sobolev-double-space}, with the norm inherited from
$H^1(\Om)\oplus H^1(\Om)$.  We use only the standard trace theorem for bounded
Lipschitz domains, the identity $\ker\Tr=H_0^1(\Om)$, the existence of a
bounded $H^1$ extension operator, and the locality and relaxation
characterization of the Cheeger energy.

First consider a Lipschitz function $F$ on $X$.  Its restrictions
$f_\pm=F|_{\overline\Om^\pm}$ are Lipschitz on $\overline\Om$, their boundary
values agree pointwise on $\partial\Om$, and hence their Sobolev traces agree.
Since the seam $\partial\Om$ has zero two-dimensional measure, locality of
minimal weak upper gradients gives that the minimal weak upper gradient of $F$
agrees almost everywhere on the two sheets with the Euclidean norms
$|\nabla f_+|$ and $|\nabla f_-|$.  Equivalently, this follows by covering the
complement of the seam by the two Euclidean polygonal charts and using the
standard identification of Cheeger and Euclidean Sobolev energies there.  Thus
\begin{equation}\label{eq:cheeger-lipschitz-double}
2\operatorname{Ch}_X(F)=
\int_\Om |\nabla f_+|^2\,dx+
\int_\Om |\nabla f_-|^2\,dx
\end{equation}
for Lipschitz $F$.

We next pass from Lipschitz functions to the full Cheeger domain by relaxation.
Let $F\in H^{1,2}(X)$.  Choose Lipschitz functions $F_k$ on $X$ such that
$F_k\to F$ in $L^2(X)$ and
$$
2\operatorname{Ch}_X(F_k)\to 2\operatorname{Ch}_X(F).
$$
Write $F_k=(f_{+,k},f_{-,k})$.  By
\eqref{eq:cheeger-lipschitz-double}, the two sequences
$f_{\pm,k}$ are bounded in $H^1(\Om)$.  After passing to weak limits on the two
sheets, these weak limits are precisely the restrictions $f_\pm$ of $F$, because
$F_k\to F$ in $L^2(X)$.  Since the trace operator on $\Om$ is continuous and
$\Tr f_{+,k}=\Tr f_{-,k}$ for every $k$, we obtain
$$
\Tr f_+=\Tr f_-
\qquad\text{on }\partial\Om.
$$
Thus $F\in\mathcal V$.  Moreover, weak lower semicontinuity of the Euclidean
Dirichlet energy gives
\begin{equation}\label{eq:double-energy-lower-bound}
\int_\Om |\nabla f_+|^2\,dx+
\int_\Om |\nabla f_-|^2\,dx
\le
2\operatorname{Ch}_X(F).
\end{equation}

Conversely, let $(f_+,f_-)\in\mathcal V$.  Put $w=f_+$.  Then
$f_+-w=0$ and $f_--w\in H_0^1(\Om)$, because
$\ker\Tr=H_0^1(\Om)$ for bounded Lipschitz domains.  Since $\Om$ is Lipschitz,
there is a bounded extension operator $H^1(\Om)\to H^1(\R^2)$.  Extending,
mollifying, and restricting back to $\overline\Om$, we obtain functions
$w_k\in C^\infty(\R^2)|_{\overline\Om}$, in particular Lipschitz on
$\overline\Om$, such that $w_k\to w$ in $H^1(\Om)$.  Choose
$a_{-,k}\in C_c^\infty(\Om)$ with
$$
a_{-,k}\to f_--w
\qquad\text{in }H^1(\Om),
$$
and put $a_{+,k}=0$.  Then
$$
F_k=(w_k+a_{+,k},\,w_k+a_{-,k})
$$
has the same pointwise boundary value $w_k|_{\partial\Om}$ on both sheets.
Each restriction is Lipschitz on $\overline\Om$.

Let $L_k$ be a common Euclidean Lipschitz constant for the two restrictions of
$F_k$.  Same-sheet Lipschitz continuity on the double is immediate from the
convexity of $\Om$.  If $x,y\in\overline\Om$ lie on opposite sheets, then for
every $q\in\partial\Om$,
\begin{equation}\label{eq:lipschitz-across-seam}
|F_k(x^+)-F_k(y^-)|
\le
L_k\bigl(|x-q|+|q-y|\bigr).
\end{equation}
Taking the infimum over $q\in\partial\Om$ in
\eqref{eq:lipschitz-across-seam} shows that $F_k$ is Lipschitz on the metric
double.  Moreover
$$
F_k\to(f_+,f_-)
\qquad\text{in }H^1(\Om)\oplus H^1(\Om).
$$
Using \eqref{eq:cheeger-lipschitz-double} and the relaxation definition of the
Cheeger energy, we obtain
\begin{equation}\label{eq:double-energy-upper-bound}
2\operatorname{Ch}_X(F)
\le
\int_\Om |\nabla f_+|^2\,dx+
\int_\Om |\nabla f_-|^2\,dx .
\end{equation}
Together, \eqref{eq:double-energy-lower-bound} and
\eqref{eq:double-energy-upper-bound} give the quadratic energy identity
$$
2\operatorname{Ch}_X(F)=
\int_\Om |\nabla f_+|^2\,dx+
\int_\Om |\nabla f_-|^2\,dx .
$$
This proves the reverse inclusion $\mathcal V\subset H^{1,2}(X)$ and the
quadratic form identity.  The bilinear formula
\eqref{eq:double-energy-form} follows by polarization.
\end{proof}

By Lemma~\ref{lem:sobolev-double}, the form
\eqref{eq:double-energy-form}, with domain
\eqref{eq:sobolev-double-space}, is a densely defined closed non-negative
symmetric form on $L^2(X,\mm_X)$.  By the representation theorem for closed
non-negative forms, it determines a unique non-negative self-adjoint operator.
We call this operator the Friedrichs operator associated with the form and denote
it by $-\Delta_X$.  Equivalently, $F\in D(-\Delta_X)$ and $-\Delta_XF=H$ if and
only if
$$
\EE_X(F,G)=\int_X H G\,d\mm_X
\qquad
\text{for all }G\in H^{1,2}(X).
$$

\begin{lemma}[Even--odd spectral splitting]\label{lem:splitting}
With multiplicities,
\begin{equation}\label{eq:spectral-splitting}
\Spec(-\Delta_X)=\Spec_N(\Om)\cup \Spec_D(\Om),
\end{equation}
where the right-hand side is the multiset union of the Neumann and Dirichlet
spectra of $\Om$.
\end{lemma}

\begin{proof}
Let $\tau:X\to X$ be the involution exchanging the two sheets.  Since $\tau$ is an
isometry preserving $\mm_X$, the induced unitary operator $UF=F\circ\tau$ on
$L^2(X)$ preserves the form domain \eqref{eq:sobolev-double-space} and the energy
form \eqref{eq:double-energy-form}.  Hence the form, and therefore the associated
self-adjoint operator, decomposes into the orthogonal direct sum of its even and
odd parts.

The maps
$$
f\longmapsto 2^{-1/2}(f,f),
\qquad
f\longmapsto 2^{-1/2}(f,-f)
$$
are unitary maps from $L^2(\Om)$ onto the even and odd subspaces of $L^2(X)$,
respectively.  On the even subspace, the form domain is $H^1(\Om)$ and the form
becomes the Neumann form $\int_\Om |\nabla f|^2\,dx$.  On the odd subspace,
trace compatibility gives $\Tr f=-\Tr f$ on $\partial\Om$, hence
$f\in H_0^1(\Om)$, and the form becomes the Dirichlet form
$\int_\Om |\nabla f|^2\,dx$ on $H_0^1(\Om)$.  The corresponding Friedrichs
operators are therefore unitarily equivalent to the Neumann and Dirichlet
Laplacians on $\Om$, which proves \eqref{eq:spectral-splitting}.
\end{proof}

Filonov's theorem~\cite[main theorem]{Filonov} is stated for Neumann eigenvalues
including the zero eigenvalue: for domains of finite measure in $\R^d$,
\begin{equation}\label{eq:filonov-inequality}
\mu_k^N(\Om)<\lambda_k^D(\Om)\qquad(k\ge1).
\end{equation}
In this indexing, our $\mu_1^N(\Om)$ is the second Neumann eigenvalue, namely the
first positive Neumann eigenvalue.

\begin{lemma}[Neumann below Dirichlet]\label{lem:ND-gap}
For every bounded connected planar Lipschitz domain $\Om$,
\begin{equation}\label{eq:ND-gap}
\mu_1^N(\Om)<\lambda_1^D(\Om).
\end{equation}
\end{lemma}

\begin{proof}
This is exactly the case $k=1$ of inequality
\eqref{eq:filonov-inequality}.
\end{proof}

Note that \eqref{eq:ND-gap} holds for the polygon chosen above.

\begin{proposition}\label{prop:limit-spectrum}
The first positive eigenvalue of $X=D\Om$ is simple and equals $\mu_1^N(\Om)$.
A corresponding normalized eigenfunction is the even double
\begin{equation}\label{eq:even-double-eigenfunction}
\Phi(x^+)=\frac{1}{\sqrt2}\phi(x),
\qquad
\Phi(x^-)=\frac{1}{\sqrt2}\phi(x).
\end{equation}
Moreover $\Phi$ has at least $m$ isolated strict local maxima on $X$, located at
the seam points corresponding to $p_1,\ldots,p_m$.
\end{proposition}

\begin{proof}
By Lemma~\ref{lem:splitting}, the first positive eigenvalue of the double is
\begin{equation}\label{eq:first-positive-double-min}
\lambda_1(X)=\min\{\mu_1^N(\Om),\lambda_1^D(\Om)\}.
\end{equation}
By Lemma~\ref{lem:ND-gap}, the minimum in \eqref{eq:first-positive-double-min} is
$\mu_1^N(\Om)$.  The eigenvalue is simple because $\mu_1^N(\Om)$ is simple by
Theorem~\ref{thm:miyamoto} and no odd Dirichlet eigenfunction occurs at that
level.

The formula \eqref{eq:even-double-eigenfunction} is the even lift of the
normalized Neumann eigenfunction $\phi$.  By \eqref{eq:phi-normalization}, this
doubled function has
$$
\int_X\Phi\,d\mm_X=0,
\qquad
\int_X\Phi^2\,d\mm_X=1.
$$
Since $\phi$ is continuous on $\overline\Om$ and the two sheet values in
\eqref{eq:even-double-eigenfunction} agree on the seam, $\Phi$ is continuous on
$X$.

Finally, a neighborhood of a seam point $p_i\in\partial\Om\subset X$ is the union
of two relative half-neighborhoods of $p_i$ in $\overline\Om$, one in each sheet.
Because $\Phi$ restricts to $2^{-1/2}\phi$ on both sheets,
Lemma~\ref{lem:strict-neighborhoods-domain} shows that the doubled neighborhoods
$DV_i$ are pairwise disjoint neighborhoods of the corresponding seam points.  By
\eqref{eq:Vi-meets-max-set}, one has
$$
\Phi(y)<\Phi(p_i)
\qquad\text{for all }y\in D V_i\setminus\{p_i\}.
$$
Thus each $p_i$ is a strict local maximum of $\Phi$ on $X$, and the pairwise
disjoint neighborhoods show that these $m$ strict local maxima are distinct and
isolated.
\end{proof}

\section{Smooth positively curved pancake approximation}\label{sec:pancakes}

We now approximate the singular doubled polygon $X=D\Om$ by smooth embedded
strictly convex surfaces in $\R^3$.

\subsection{Smooth inner approximations of the polygon}

In this subsection, $d_H^{\R^2}$ denotes Euclidean Hausdorff distance in the
plane, and $|A|$ denotes planar Lebesgue measure for measurable sets
$A\subset\R^2$.

Write the convex polygon as
\begin{equation}\label{eq:polygon-halfspace}
\Om=\bigcap_{\alpha=1}^{L}\{x\in\R^2:\ell_\alpha(x)\le0\},
\end{equation}
where $L$ is the number of defining half-spaces, each $\ell_\alpha$ is affine
and nonconstant, and where the outward normals $d\ell_\alpha$ span $\R^2$.  For
$\delta>0$ set
\begin{equation}\label{eq:psi-delta-and-omega-delta}
\psi_\delta(x)=\sum_{\alpha=1}^{L} \exp(\ell_\alpha(x)/\delta),
\qquad
\Om_\delta=\{x\in\R^2:\psi_\delta(x)\le1\}.
\end{equation}
For all sufficiently small $\delta$, $\Om_\delta$ has non-empty interior.

\begin{lemma}\label{lem:Omega-delta}
For all sufficiently small $\delta>0$, $\Om_\delta$ is a bounded smooth strictly
convex domain, $\Om_\delta\subset\Om$, and
\begin{equation}\label{eq:omega-delta-convergence}
d_H^{\R^2}(\overline{\Om_\delta},\overline\Om)\to0,
\qquad
d_H^{\R^2}(\partial\Om_\delta,\partial\Om)\to0,
\qquad
|\Om\setminus\Om_\delta|\to0
\quad\text{as }\delta\downarrow0.
\end{equation}
\end{lemma}

\begin{proof}
Choose $x_*\in\Om^\circ$.  Since $\ell_\alpha(x_*)<0$ for all $\alpha$, one has
$\psi_\delta(x_*)<1$ for all sufficiently small $\delta$.  Thus $\Om_\delta$ has
non-empty interior.

The Hessian is
\begin{equation}\label{eq:hessian-psi-delta}
D^2\psi_\delta(x)=\frac{1}{\delta^2}\sum_{\alpha=1}^{L}
\exp(\ell_\alpha(x)/\delta)\,d\ell_\alpha\otimes d\ell_\alpha .
\end{equation}
Because the covectors $d\ell_\alpha$ span $\R^2$ and the coefficients in
\eqref{eq:hessian-psi-delta} are positive, $D^2\psi_\delta$ is positive definite
at every point.  Thus $\psi_\delta$ is smooth and strictly convex, and every
sublevel set $\{\psi_\delta\le c\}$ is strictly convex.  On the level set
$\psi_\delta=1$, the gradient cannot vanish: if $\nabla\psi_\delta=0$ at a point
of that level set, strict convexity would make that point the unique global
minimizer, whereas the existence of $x_*$ with $\psi_\delta(x_*)<1$ shows that
the minimum is strictly less than $1$.  Therefore $\partial\Om_\delta$ is smooth.
Since $\Om_\delta\subset\Om$ will be shown below and $\Om$ is bounded,
$\Om_\delta$ is bounded.

If $x\in\Om_\delta$, then each exponential term in \eqref{eq:psi-delta-and-omega-delta}
is at most $1$, hence each $\ell_\alpha(x)\le0$, so $x\in\Om$ by
\eqref{eq:polygon-halfspace}.  Thus $\Om_\delta\subset\Om$.

It remains to prove Hausdorff convergence.  Put
$$
C_0:=\min_{1\le\alpha\le L}|d\ell_\alpha|>0.
$$
If $x\in\Om$ and $\dist_{\R^2}(x,\partial\Om)\ge r$, then, for each $\alpha$, the exterior
half-space $\{\ell_\alpha>0\}$ is contained in $\R^2\setminus\Om$.  Hence
$$
\frac{-\ell_\alpha(x)}{|d\ell_\alpha|}
=\dist_{\R^2}(x,\{\ell_\alpha=0\})
\ge \dist_{\R^2}(x,\partial\Om)
\ge r.
$$
Consequently, for the inner parallel set
$$
\Om^{-r}:=\{x\in\Om:\dist_{\R^2}(x,\partial\Om)\ge r\},
$$
one has the linear error bound
\begin{equation}\label{eq:linear-error-bound}
\ell_\alpha(x)\le -C_0 r
\quad\text{for all }x\in\Om^{-r},\;\alpha=1,\ldots,L.
\end{equation}
Choose
\begin{equation}\label{eq:r-delta}
r_\delta=\frac{2\delta}{C_0}\log L.
\end{equation}
Then, for $x\in\Om^{-r_\delta}$, equations \eqref{eq:linear-error-bound} and
\eqref{eq:r-delta} give
$$
\psi_\delta(x)\le L\exp(-C_0r_\delta/\delta)=L^{-1}<1.
$$
Hence
\begin{equation}\label{eq:inner-outer-inclusion}
\Om^{-r_\delta}\subset\Om_\delta\subset\Om.
\end{equation}
Since $r_\delta\to0$, \eqref{eq:inner-outer-inclusion} implies
$d_H^{\R^2}(\overline{\Om_\delta},\overline\Om)\to0$.

We now prove boundary convergence.  Every point of $\partial\Om_\delta$ lies
within distance $r_\delta$ of $\partial\Om$, otherwise it would belong to
$\Om^{-r_\delta}\subset\{\psi_\delta<1\}$ by \eqref{eq:inner-outer-inclusion}.

Conversely, put
$$
s_*:=\min_{1\le\alpha\le L}\frac{-\ell_\alpha(x_*)}{|d\ell_\alpha|}>0.
$$
For $p\in\partial\Om$ define
$$
t_\delta=\frac{2r_\delta}{s_*},
\qquad
q_{\delta,p}:=(1-t_\delta)p+t_\delta x_* .
$$
For all sufficiently small $\delta$, $0<t_\delta<1$.  By the half-space
inequalities,
$$
\ell_\alpha(q_{\delta,p})
=(1-t_\delta)\ell_\alpha(p)+t_\delta\ell_\alpha(x_*)
\le t_\delta\ell_\alpha(x_*)
\le -2r_\delta |d\ell_\alpha|
$$
for every $\alpha$.  Thus $\dist_{\R^2}(q_{\delta,p},\partial\Om)\ge2r_\delta$, and in
particular $q_{\delta,p}\in\Om^{-r_\delta}\subset\{\psi_\delta<1\}$.  On the
other hand, since $p\in\partial\Om$, at least one defining inequality satisfies
$\ell_\alpha(p)=0$, hence $\psi_\delta(p)\ge1$.  By continuity, the segment
$[p,q_{\delta,p}]$ meets $\partial\Om_\delta$ at some point $b_{\delta,p}$.  The
estimate
$$
|b_{\delta,p}-p|\le |q_{\delta,p}-p|
=t_\delta |x_*-p|
\le \frac{2\diam(\Om)}{s_*}\,r_\delta
$$
is uniform in $p$.  This proves $d_H^{\R^2}(\partial\Om_\delta,\partial\Om)\to0$.
Finally, $\Om\setminus\Om_\delta$ is contained in the boundary layer
$\Om\setminus\Om^{-r_\delta}$, whose Lebesgue measure tends to $0$ as
$r_\delta\downarrow0$.  This completes the proof of \eqref{eq:omega-delta-convergence}.
\end{proof}

\begin{lemma}\label{lem:doubles-converge}
The metric-measure doubles satisfy
\begin{equation}\label{eq:doubles-converge}
(D\Om_\delta,d_{D\Om_\delta},\mm_{D\Om_\delta})
\longrightarrow
(D\Om,d_{D\Om},\mm_{D\Om})
\end{equation}
in measured Gromov--Hausdorff sense as $\delta\downarrow0$.
\end{lemma}

\begin{proof}
For any compact convex set $C\subset\R^2$ with non-empty interior, distances in
the metric double $DC$ are given by the following formulas, where $|\cdot|$
denotes the Euclidean norm in $\R^2$:
\begin{subequations}
\begin{align}
d_{DC}(x^\sigma,y^\sigma)&=|x-y|,
\qquad \sigma\in\{+,-\},\label{eq:double-distance-same-sheet}\\
d_{DC}(x^+,y^-)&=\inf_{q\in\partial C}\bigl(|x-q|+|q-y|\bigr).\label{eq:double-distance-opposite-sheet}
\end{align}
\end{subequations}
Formula \eqref{eq:double-distance-same-sheet} follows because projection to the
plane is length non-increasing and the straight segment in a convex sheet realizes
$|x-y|$.  Formula \eqref{eq:double-distance-opposite-sheet} follows because any
curve from one sheet to the other must cross the seam, while the broken segment
$xq\cup qy$ gives the displayed upper bound.

Let
$$
a_\delta=d_H^{\R^2}(\overline{\Om_\delta},\overline\Om),
\qquad
b_\delta=d_H^{\R^2}(\partial\Om_\delta,\partial\Om),
$$
so $a_\delta,b_\delta\to0$ by Lemma~\ref{lem:Omega-delta}.  Choose a Borel map
$Q_\delta:\partial\Om_\delta\to\partial\Om$ with
\begin{equation}\label{eq:Q-delta-bound}
|Q_\delta(a)-a|\le b_\delta+\delta
\qquad(a\in\partial\Om_\delta).
\end{equation}
Define a seam-compatible map
$$
\mathcal F_\delta:D\Om_\delta\to D\Om
$$
as follows: if $x\in\Om_\delta^\circ$, set $\mathcal F_\delta(x^\pm)=x^\pm$; if
$a\in\partial\Om_\delta$ is a seam point, set $\mathcal F_\delta(a)=Q_\delta(a)$, viewed
as a seam point of $D\Om$.  This map need not be continuous near the seam, but
continuity is not required for a Gromov--Hausdorff approximation; only distortion
and density are used.

The image $\mathcal F_\delta(D\Om_\delta)$ is $(a_\delta+b_\delta+
\delta)$-dense in $D\Om$.  Indeed, points in $\Om\setminus\Om_\delta$ are within
$a_\delta$ of $\overline{\Om_\delta}$ on the corresponding sheet, and if the
nearest point is on $\partial\Om_\delta$, its image is changed by at most
$b_\delta+\delta$ by \eqref{eq:Q-delta-bound}.

We next estimate distortion.  Same-sheet distances between interior points are
preserved exactly.  If one or both points are seam points, replacing
$a\in\partial\Om_\delta$ by $Q_\delta(a)$ changes the relevant Euclidean length
by at most $b_\delta+\delta$ per seam point, again by \eqref{eq:Q-delta-bound}.
Hence same-sheet distances are changed by at most $2(b_\delta+\delta)$.

For opposite-sheet distances, use \eqref{eq:double-distance-opposite-sheet}.  For
$x,y\in\overline{\Om_\delta}$,
$$
\left|
\inf_{q\in\partial\Om_\delta}(|x-q|+|q-y|)
-
\inf_{q\in\partial\Om}(|x-q|+|q-y|)
\right|
\le 2b_\delta.
$$
Indeed, by the definition of Hausdorff distance, for every
$q\in\partial\Om_\delta$ there is $q'\in\partial\Om$ with
$|q-q'|\le b_\delta$, and conversely.  The triangle inequality gives
$|x-q'|+|q'-y|\le |x-q|+|q-y|+2b_\delta$ and the reverse comparison in the same
way.  If $x$ or $y$ itself is a seam point, the additional replacement by $Q_\delta$
changes the two Euclidean terms by at most $2(b_\delta+\delta)$.  Therefore
\begin{equation}\label{eq:calF-delta-distortion}
\sup_{p,q\in D\Om_\delta}
\bigl|d_{D\Om}(\mathcal F_\delta p,\mathcal F_\delta q)-d_{D\Om_\delta}(p,q)\bigr|
\le 4b_\delta+2\delta.
\end{equation}
The density estimate and \eqref{eq:calF-delta-distortion} show that $\mathcal F_\delta$ is a
Gromov--Hausdorff approximation.

For the measures, $\mm_{D\Om_\delta}$ is the sum of two copies of planar
Lebesgue measure on $\Om_\delta$, and the seam has zero two-dimensional measure.
The push-forward $(\mathcal F_\delta)_\#\mm_{D\Om_\delta}$ is therefore the sum of two
copies of Lebesgue measure restricted to $\Om_\delta\subset\Om$, viewed inside
the corresponding sheets of $D\Om$.  Since $\chi_{\Om_\delta}\to\chi_\Om$ in
$L^1(\R^2)$ by Lemma~\ref{lem:Omega-delta}, these measures converge weakly to
$\mm_{D\Om}$.  This proves \eqref{eq:doubles-converge}.
\end{proof}

\subsection{Elliptic pancakes over a smooth convex base}

Fix $\delta>0$ small.  For $\eps>0$, define
\begin{equation}\label{eq:pancake-body-and-surface}
C_{\eps,\delta}:=\left\{(x,z)\in\R^2\times\R:
\psi_\delta(x)+\frac{z^2}{\eps^2}\le1\right\},
\qquad
M_{\eps,\delta}:=\partial C_{\eps,\delta}.
\end{equation}

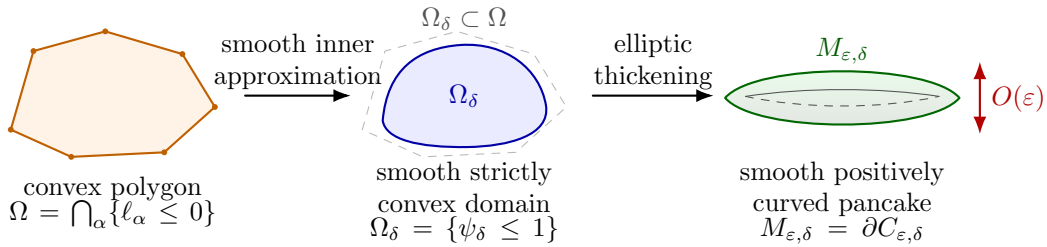
\begin{figure}[htbp]
\centering
\begin{tikzpicture}[
  >=Latex,
  every node/.style={font=\small},
  line cap=round,
  line join=round
]

\tikzset{
  polygonstyle/.style={draw=orange!75!black, thick, fill=orange!10},
  smoothstyle/.style={draw=blue!65!black, thick, fill=blue!8},
  surfacestyle/.style={draw=green!40!black, thick}
}

\begin{scope}[shift={(0,0)}]
  \coordinate (P1) at (-1.35,-0.42);
  \coordinate (P2) at (-1.05, 0.62);
  \coordinate (P3) at (-0.10, 0.88);
  \coordinate (P4) at ( 0.92, 0.58);
  \coordinate (P5) at ( 1.35,-0.12);
  \coordinate (P6) at ( 0.68,-0.72);
  \coordinate (P7) at (-0.55,-0.78);

  \filldraw[polygonstyle]
    (P1)--(P2)--(P3)--(P4)--(P5)--(P6)--(P7)--cycle;

  \foreach \P in {P1,P2,P3,P4,P5,P6,P7}{
    \fill[orange!75!black] (\P) circle (1.1pt);
  }

  \node[align=center, text width=3.4cm] at (0,-1.38)
    {convex polygon\\[-1mm]$\Om=\bigcap_\alpha\{\ell_\alpha\le0\}$};
\end{scope}

\draw[->, thick] (1.75,0.03) -- (3.15,0.03);
\node[align=center] at (2.45,0.48)
  {smooth inner\\approximation};

\begin{scope}[shift={(4.65,0)}]
  \coordinate (Q1) at (-1.35,-0.42);
  \coordinate (Q2) at (-1.05, 0.62);
  \coordinate (Q3) at (-0.10, 0.88);
  \coordinate (Q4) at ( 0.92, 0.58);
  \coordinate (Q5) at ( 1.35,-0.12);
  \coordinate (Q6) at ( 0.68,-0.72);
  \coordinate (Q7) at (-0.55,-0.78);

  \draw[gray!55, dashed]
    (Q1)--(Q2)--(Q3)--(Q4)--(Q5)--(Q6)--(Q7)--cycle;

  \filldraw[smoothstyle]
    (-1.08,-0.30)
    .. controls (-1.02, 0.38) and (-0.62, 0.66) .. (-0.05, 0.70)
    .. controls ( 0.55, 0.74) and ( 1.00, 0.38) .. ( 1.08,-0.10)
    .. controls ( 1.14,-0.50) and ( 0.58,-0.64) .. ( 0.05,-0.65)
    .. controls (-0.62,-0.66) and (-1.04,-0.55) .. (-1.08,-0.30)
    -- cycle;

  \node[blue!65!black] at (0,0.03) {$\Om_\delta$};
  \node[gray!65!black] at (0,1.05) {$\Om_\delta\subset\Om$};

  \node[align=center, text width=3.4cm] at (0,-1.38)
    {smooth strictly\\convex domain\\[-1mm]$\Om_\delta=\{\psi_\delta\le1\}$};
\end{scope}

\draw[->, thick] (6.35,0.03) -- (7.95,0.03);
\node[align=center] at (7.15,0.48)
  {elliptic\\thickening};

\begin{scope}[shift={(9.65,0)}]
  \shade[top color=green!8, bottom color=gray!22]
    (-1.55,0)
    .. controls (-1.14, 0.47) and ( 1.14, 0.47) .. ( 1.55,0)
    .. controls ( 1.14,-0.47) and (-1.14,-0.47) .. (-1.55,0)
    -- cycle;

  \draw[surfacestyle]
    (-1.55,0)
    .. controls (-1.14, 0.47) and ( 1.14, 0.47) .. ( 1.55,0)
    .. controls ( 1.14,-0.47) and (-1.14,-0.47) .. (-1.55,0)
    -- cycle;

  \draw[gray!65!black, dashed]
    (-1.28,0.02) .. controls (-0.50,-0.15) and (0.50,-0.15) .. (1.28,0.02);
  \draw[gray!65!black]
    (-1.28,0.02) .. controls (-0.50, 0.14) and (0.50, 0.14) .. (1.28,0.02);

  \draw[red!70!black, <->, thick] (1.82,-0.48) -- (1.82,0.48);
  \node[red!70!black, right] at (1.84,0) {$O(\eps)$};

  \node[green!35!black] at (0,0.63) {$M_{\eps,\delta}$};

  \node[align=center, text width=3.4cm] at (0,-1.38)
    {smooth positively\\curved pancake\\[-1mm]$M_{\eps,\delta}=\partial C_{\eps,\delta}$};
\end{scope}

\end{tikzpicture}
\caption{Schematic of the pancake approximation.  A convex polygon $\Om$ is first
replaced by smooth strictly convex inner domains $\Om_\delta$, and then by thin
strictly convex surfaces $M_{\eps,\delta}=\partial C_{\eps,\delta}$ with positive
Gaussian curvature.}
\label{fig:pancake-approximation-schematic}
\end{figure}

\begin{lemma}\label{lem:positive-curvature}
For every sufficiently small $\delta>0$ and every $\eps>0$,
$M_{\eps,\delta}$ is a smooth embedded two-sphere with positive Gaussian
curvature.
\end{lemma}

\begin{proof}
Let
$$
G_{\eps,\delta}(x,z)=\psi_\delta(x)+\frac{z^2}{\eps^2}.
$$
Then
\begin{equation}\label{eq:hessian-G-eps-delta}
D^2G_{\eps,\delta}=\begin{pmatrix}
D^2\psi_\delta&0\\[2mm]
0&2\eps^{-2}
\end{pmatrix}>0.
\end{equation}
The level set $G_{\eps,\delta}=1$ is regular.  Indeed, if
$\nabla G_{\eps,\delta}=0$ on that level set, strict convexity would place a
global minimizer on the boundary of a sublevel set whose interior is non-empty,
impossible.  Hence $M_{\eps,\delta}$ is smooth.  It is compact and bounds the
strictly convex body $C_{\eps,\delta}$ from \eqref{eq:pancake-body-and-surface},
so it is an embedded topological sphere.  Since it is a smooth embedded closed
surface homeomorphic to $\mathbb S^2$, it is diffeomorphic to $\mathbb S^2$; equivalently,
for a smooth strictly convex body the outward Gauss map gives a smooth
diffeomorphism with the unit sphere.

Take the outward unit normal
$$
\nu=\frac{\nabla G_{\eps,\delta}}{|\nabla G_{\eps,\delta}|}.
$$
For tangent vectors $v\in T M_{\eps,\delta}$, the second fundamental form in the
convention $\mathrm{II}(v,v)=\langle D_v\nu,v\rangle$ is
\begin{equation}\label{eq:second-fundamental-level-set}
\mathrm{II}(v,v)=\frac{D^2G_{\eps,\delta}(v,v)}{|\nabla G_{\eps,\delta}|}.
\end{equation}
Since $D^2G_{\eps,\delta}$ is positive definite by
\eqref{eq:hessian-G-eps-delta}, $\mathrm{II}$ is positive definite in this
convention.  With the opposite sign convention for the second fundamental form it
would be negative definite; in either convention the two principal curvatures
have the same nonzero sign, and their product, the Gaussian curvature, is
positive.
\end{proof}

\begin{lemma}\label{lem:pancake-convergence-fixed-delta}
For each fixed sufficiently small $\delta>0$,
\begin{equation}\label{eq:pancake-fixed-delta-convergence}
(M_{\eps,\delta},d_{M_{\eps,\delta}},\mm_{M_{\eps,\delta}})
\longrightarrow
(D\Om_\delta,d_{D\Om_\delta},\mm_{D\Om_\delta})
\end{equation}
in measured Gromov--Hausdorff sense as $\eps\downarrow0$.
\end{lemma}

\begin{proof}
Set
$$
f_\delta(x)=\sqrt{1-\psi_\delta(x)},
\qquad x\in\overline{\Om_\delta}.
$$
Because $\psi_\delta$ is convex, $1-\psi_\delta$ is concave and nonnegative on
$\Om_\delta$; since $t\mapsto\sqrt t$ is increasing and concave on
$[0,\infty)$, $f_\delta$ is concave.  Also $0\le f_\delta\le1$.
Note that $M_{\eps,\delta}$ is the union of the two graphs
\begin{equation}\label{eq:pancake-two-graphs}
z=\eps f_\delta(x),
\qquad
z=-\eps f_\delta(x),
\qquad
x\in\overline{\Om_\delta}.
\end{equation}
The graph parametrization in \eqref{eq:pancake-two-graphs} is singular at
$\partial\Om_\delta$, but Lemma~\ref{lem:positive-curvature} proves that the
level surface itself is smooth there.

Define
$$
\pi_{\eps,\delta}:M_{\eps,\delta}\to D\Om_\delta
$$
by horizontal projection, sending the upper graph to the positive sheet and the
lower graph to the negative sheet.  This is well-defined on the seam because
$f_\delta=0$ on $\partial\Om_\delta$, and the two boundary copies are identified
in the double.  Horizontal projection does not increase the length of rectifiable
curves.  Therefore
\begin{equation}\label{eq:pancake-projection-lower-bound}
d_{D\Om_\delta}(\pi_{\eps,\delta}(p),\pi_{\eps,\delta}(q))
\le d_{M_{\eps,\delta}}(p,q)
\qquad\text{for all }p,q\in M_{\eps,\delta}.
\end{equation}

Conversely, suppose first that $p,q$ lie on the same sheet and that
$\pi_{\eps,\delta}(p)=x^\sigma$, $\pi_{\eps,\delta}(q)=y^\sigma$.  The Euclidean
segment $\gamma(t)=(1-t)x+ty$ is contained in $\Om_\delta$.  Since
$f_\delta\circ\gamma$ is concave, it has bounded variation, and the length of its
lift to the corresponding graph satisfies
\begin{equation}\label{eq:same-sheet-lift-length}
\mathcal L\bigl(t\mapsto(\gamma(t),\pm\eps f_\delta(\gamma(t)))\bigr)
\le |x-y|+\eps\operatorname{Var}(f_\delta\circ\gamma).
\end{equation}
Because $0\le f_\delta\le1$, the total variation of the one-dimensional concave
function $f_\delta\circ\gamma$ is at most $2$.  Thus the lifted segment in
\eqref{eq:same-sheet-lift-length} has length at most $|x-y|+2\eps$.

If $p$ and $q$ lie on opposite sheets, then, for arbitrary $\rho>0$, choose
$a\in\partial\Om_\delta$ such that
$$
|x-a|+|a-y|\le d_{D\Om_\delta}(x^+,y^-)+\rho.
$$
Lift $x\to a$ on the upper graph and $a\to y$ on the lower graph.  By the same
estimate, the total lifted length is at most
$ d_{D\Om_\delta}(x^+,y^-)+4\eps+\rho $.  Letting $\rho\downarrow0$ gives
\begin{equation}\label{eq:pancake-projection-upper-bound}
d_{M_{\eps,\delta}}(p,q)
\le d_{D\Om_\delta}(\pi_{\eps,\delta}(p),\pi_{\eps,\delta}(q))+4\eps.
\end{equation}
Equations \eqref{eq:pancake-projection-lower-bound} and
\eqref{eq:pancake-projection-upper-bound} show that $\pi_{\eps,\delta}$ is onto
and has distortion at most $4\eps$.  This is Gromov--Hausdorff convergence.

It remains to prove measure convergence.  On one graph, away from the boundary,
the area density is
\begin{equation}\label{eq:area-density}
dA_{\eps,\delta}=\sqrt{1+\eps^2|\nabla f_\delta|^2}\,dx.
\end{equation}
Although $f_\delta$ is not Lipschitz up to $\partial\Om_\delta$, the formula
\eqref{eq:area-density} is justified by applying the graph area formula first on
the compact exhaustion
$$
\Om_{\delta,s}:=\{x\in\Om_\delta:\dist_{\R^2}(x,\partial\Om_\delta)\ge s\},
$$
where $f_\delta$ is smooth, and then letting $s\downarrow0$.  The exhausted graph
pieces increase to the whole graph up to the seam, the seam itself has zero
two-dimensional area, and monotone convergence of the areas applies once the
finite collar contribution is verified by the estimate below.

We claim that
\begin{equation}\label{eq:grad-f-integrable}
|\nabla f_\delta|\in L^1(\Om_\delta).
\end{equation}
Away from the boundary this is clear.  Near $\partial\Om_\delta$, the boundary is
smooth and $|\nabla\psi_\delta|$ is bounded above and below by positive
constants.  In a normal collar, if $r=\dist_{\R^2}(x,\partial\Om_\delta)$, Taylor
expansion gives $1-\psi_\delta(x)\simeq c r$ and
$|\nabla\psi_\delta(x)|\le C$.  Hence
$$
|\nabla f_\delta(x)|=
\frac{|\nabla\psi_\delta(x)|}{2\sqrt{1-\psi_\delta(x)}}
\le C r^{-1/2},
$$
which is integrable in a two-dimensional collar.  This proves
\eqref{eq:grad-f-integrable}.  Therefore, for $0<\eps\le1$,
$$
0\le \sqrt{1+\eps^2|\nabla f_\delta|^2}-1\le \eps |\nabla f_\delta|
\le |\nabla f_\delta|,
$$
and dominated convergence, using \eqref{eq:grad-f-integrable}, gives
$$
\sqrt{1+\eps^2|\nabla f_\delta|^2}\to1
\quad\text{in }L^1(\Om_\delta).
$$
The same holds on the lower graph.  Thus
$(\pi_{\eps,\delta})_\#\mm_{M_{\eps,\delta}}$ converges weakly to the sum
of two copies of Lebesgue measure on $\Om_\delta$, namely to
$\mm_{D\Om_\delta}$.  Together with the Gromov--Hausdorff convergence above,
this proves \eqref{eq:pancake-fixed-delta-convergence}.
\end{proof}

\begin{proposition}[Pancake approximation]\label{prop:pancake-approx}
There exists a sequence of smooth embedded surfaces $M_j\subset\R^3$, each
diffeomorphic to $\mathbb S^2$ and having positive Gaussian curvature, such that
\begin{equation}\label{eq:pancake-sequence-convergence}
(M_j,d_{M_j},\mm_{M_j})
\longrightarrow
(X,d_X,\mm_X)
\end{equation}
in measured Gromov--Hausdorff sense.
\end{proposition}

\begin{proof}
Choose any sequence $\delta_j\downarrow0$.  By
Lemma~\ref{lem:pancake-convergence-fixed-delta}, for each $j$ one may choose
$\eps_j>0$ so small that the measured Gromov--Hausdorff distance from
$M_{\eps_j,\delta_j}$ to $D\Om_{\delta_j}$ is less than $1/j$.  Set
\begin{equation}\label{eq:Mj-definition}
M_j=M_{\eps_j,\delta_j}.
\end{equation}
By Lemma~\ref{lem:positive-curvature}, each $M_j$ in \eqref{eq:Mj-definition} is
a smooth positively curved embedded two-sphere.  Lemma~\ref{lem:doubles-converge}
and the triangle inequality for measured Gromov--Hausdorff convergence give
\eqref{eq:pancake-sequence-convergence}.
\end{proof}

\section{Spectral convergence and persistence of maxima}\label{sec:spectral}

\subsection{The spectral-convergence tool}

We use Honda's compactness theorem for spectral data on compact
finite-dimensional \(\RCD\) spaces.  The formulation below is the compact case of
\cite[Theorem~3.7]{Honda}, with convergence of spectral data understood in the
sense of \cite[Definition~3.6]{Honda}.

\begin{theorem}[Honda spectral-data compactness]\label{thm:honda}
Let
\[
(Y_j,d_j,\mm_j)\longrightarrow (Y,d,\mm)
\]
be a measured Gromov--Hausdorff convergent sequence of compact
finite-dimensional \(\RCD(\kappa,n)\) spaces.  Assume that, after discarding
finitely many terms, there are constants \(d_0,v_0\ge1\) such that
\[
d_0^{-1}\le \diam(Y_j,d_j)\le d_0,\qquad
d_0^{-1}\le \diam(Y,d)\le d_0,
\]
and
\[
v_0^{-1}\le \mm_j(Y_j)\le v_0,\qquad
v_0^{-1}\le \mm(Y)\le v_0 .
\]
Choose Borel maps
\[
\mathcal F_j:Y_j\to Y
\]
realizing the measured Gromov--Hausdorff convergence, in the sense that
\(\mathcal F_j\) is a \(\gamma_j\)-Gromov--Hausdorff approximation,
\(\gamma_j\downarrow0\), and
\[
(\mathcal F_j)_\#\mm_j\rightharpoonup\mm .
\]
For compact spaces with discrete spectrum, let the eigenvalues be indexed as in
\eqref{eq:eigenvalue-indexing}.  Then
\begin{equation}\label{eq:honda-eigenvalue-convergence}
\lambda_k(Y_j)\to\lambda_k(Y)
\qquad(k=0,1,2,\ldots),
\end{equation}
with eigenvalues counted with multiplicity.

Moreover, if
\[
a_j=(\varphi^j_0,\varphi^j_1,\ldots)
\]
is any sequence of orthonormal spectral data on \(Y_j\), then, after passing to a
subsequence, there is an orthonormal spectral datum
\[
a=(\varphi_0,\varphi_1,\ldots)
\]
on \(Y\) such that, for each fixed \(\ell\),
\[
\varphi^j_\ell(y_j)\to \varphi_\ell(y)
\qquad\text{whenever}\qquad
\mathcal F_j(y_j)\to y .
\]
In particular, for each fixed \(\ell\),
\begin{equation}\label{eq:honda-spectral-data-uniform}
\sup_{y_j\in Y_j}
\left|
\varphi^j_\ell(y_j)-\varphi_\ell(\mathcal F_j(y_j))
\right|
\longrightarrow0 .
\end{equation}
\end{theorem}

\begin{corollary}[Simple limiting eigenvalue]\label{cor:honda-simple}
In the setting of Theorem~\ref{thm:honda}, assume that \(\lambda_k(Y)\) is
simple.  Then \(\lambda_k(Y_j)\) is simple for all sufficiently large \(j\).

Let \(u_j\) be an \(L^2(\mm_j)\)-normalized \(\lambda_k(Y_j)\)-eigenfunction, and
let \(u\) be an \(L^2(\mm)\)-normalized \(\lambda_k(Y)\)-eigenfunction.  After
passing to a subsequence and multiplying \(u_j\) by signs, there are numbers
\(\beta_j\downarrow0\) such that
\begin{equation}\label{eq:honda-uniform-with-maps}
\sup_{y_j\in Y_j}
\left|
u_j(y_j)-u(\mathcal F_j(y_j))
\right|
\le \beta_j .
\end{equation}
\end{corollary}

\begin{proof}
Since \(\lambda_k(Y)\) is simple, it is separated from the adjacent eigenvalues:
\[
\lambda_{k-1}(Y)<\lambda_k(Y)<\lambda_{k+1}(Y),
\]
with the left inequality omitted when \(k=0\).  By
\eqref{eq:honda-eigenvalue-convergence}, the same strict separation holds for
\(Y_j\) for all sufficiently large \(j\).  Hence \(\lambda_k(Y_j)\) is simple
for all sufficiently large \(j\).

For these \(j\), complete \(u_j\) to an orthonormal spectral datum
\[
a_j=(\varphi^j_0,\varphi^j_1,\ldots)
\]
on \(Y_j\), choosing \(\varphi^j_k=u_j\).  By
Theorem~\ref{thm:honda}, after passing to a subsequence, these spectral data
converge to an orthonormal spectral datum
\[
a=(\varphi_0,\varphi_1,\ldots)
\]
on \(Y\).  In particular, \(\varphi_k\) is an \(L^2(\mm)\)-normalized
\(\lambda_k(Y)\)-eigenfunction.  Since the limiting eigenspace is
one-dimensional, \(\varphi_k=u\) or \(\varphi_k=-u\).  Replacing \(u_j\) by
\(-u_j\) along the subsequence if necessary, we may assume that the limit is
\(u\).  Then \eqref{eq:honda-spectral-data-uniform} gives
\eqref{eq:honda-uniform-with-maps}.
\end{proof}

We remark that
Honda's definition of measured Gromov--Hausdorff approximation is written using
normalized reference measures together with convergence of the total masses; see
\cite[Definition~3.4]{Honda}.  In the finite-measure convention used here this is
equivalent to weak convergence of the unnormalized push-forward measures.  A
constant rescaling of the reference measure leaves Rayleigh quotients, and hence
eigenvalues, unchanged; it only rescales \(L^2\)-normalized eigenfunctions.

\begin{lemma}\label{lem:RCD}
The sequence \(M_j\) from Proposition~\ref{prop:pancake-approx} and the limit
\(X=D\Om\) satisfy the hypotheses of Theorem~\ref{thm:honda} with
\(\kappa=0\) and \(n=2\), after discarding finitely many terms.
\end{lemma}

\begin{proof}
Let \(g_j\) be the induced Riemannian metric on \(M_j\).  Each \(M_j\) is a
smooth closed two-dimensional Riemannian manifold with positive Gaussian
curvature.  In dimension two,
\[
\operatorname{Ric}_{g_j}=\mathsf K_{g_j}g_j ,
\]
and hence \(\operatorname{Ric}_{g_j}\ge0\).  Therefore each \(M_j\), equipped with
its Riemannian distance and area measure, is an \(\RCD(0,2)\) space.

The convex polygon \(\Om\), with its intrinsic Euclidean metric, is a compact
Alexandrov space of curvature at least \(0\) with boundary.  Equivalently, at a
vertex of interior angle \(\theta\le\pi\), the double has cone angle
\(2\theta\le2\pi\), while along non-vertex seam points the total angle is
\(2\pi\).  Thus the polyhedral curvature of the double is non-negative.  By
Petrunin's gluing theorem, in the form used in
Kapovitch--Ketterer--Sturm~\cite[Theorem~2.13]{KKS}, the metric double
\(X=D\Om\) is a compact two-dimensional Alexandrov space of curvature at least
\(0\).  Kapovitch--Ketterer--Sturm~\cite[Corollary~2.10]{KKS} then implies that
\((X,d_X,\mm_X)\) is an \(\RCD(0,2)\) space.

Finally, Proposition~\ref{prop:pancake-approx} gives measured
Gromov--Hausdorff convergence \(M_j\to X\).  The limit \(X=D\Om\) is not a
single point and satisfies \(0<\mm_X(X)<\infty\).  Hence the diameters and total
measures of \(M_j\) are, for all sufficiently large \(j\), bounded above and
below away from zero.  Thus the spaces \(M_j\) and \(X\) lie, after discarding
finitely many terms, in a common compact class of \(\RCD(0,2)\) spaces of the
type required in Theorem~\ref{thm:honda}.
\end{proof}

\begin{proposition}\label{prop:eigenfunction-convergence}
For all sufficiently large \(j\), \(\lambda_1(M_j)\) is simple.  Let \(u_j\) be an
\(L^2(\mm_{M_j})\)-normalized \(\lambda_1(M_j)\)-eigenfunction.  After passing to
a subsequence and multiplying \(u_j\) by signs, one has
\begin{equation}\label{eq:lambda1-convergence-to-limit}
\lambda_1(M_j)\to\lambda_1(X)=\mu_1^N(\Om),
\end{equation}
and there exist measured Gromov--Hausdorff approximation maps
\[
\mathcal F_j:M_j\to X
\]
and numbers \(\beta_j\downarrow0\) such that
\begin{equation}\label{eq:eigenfunction-uniform-convergence-to-Phi}
\sup_{q\in M_j}
\left|
u_j(q)-\Phi(\mathcal F_j(q))
\right|
\le \beta_j .
\end{equation}
Here \(\Phi\) is the normalized even double from
Proposition~\ref{prop:limit-spectrum}.
\end{proposition}

\begin{proof}
Choose measured Gromov--Hausdorff approximation maps
\[
\mathcal F_j:M_j\to X
\]
realizing the convergence in Proposition~\ref{prop:pancake-approx}.  By
Lemma~\ref{lem:RCD}, Theorem~\ref{thm:honda} applies to this convergence.
Therefore \eqref{eq:honda-eigenvalue-convergence} gives
\[
\lambda_k(M_j)\to\lambda_k(X)
\qquad(k=0,1,2,\ldots).
\]
Taking \(k=1\) and using Proposition~\ref{prop:limit-spectrum}, we obtain
\[
\lambda_1(M_j)\to\lambda_1(X)=\mu_1^N(\Om),
\]
which is \eqref{eq:lambda1-convergence-to-limit}.

By Proposition~\ref{prop:limit-spectrum}, the eigenvalue \(\lambda_1(X)\) is
simple.  Corollary~\ref{cor:honda-simple}, applied with \(k=1\) and
\(Y_j=M_j\), \(Y=X\), therefore implies that \(\lambda_1(M_j)\) is simple for all
sufficiently large \(j\).  Applying the same corollary to the normalized
eigenfunctions \(u_j\), and choosing the sign of \(u_j\) so that the limit is
\(\Phi\), gives
\[
\sup_{q\in M_j}
\left|
u_j(q)-\Phi(\mathcal F_j(q))
\right|
\to0 .
\]
This is exactly \eqref{eq:eigenfunction-uniform-convergence-to-Phi}.
\end{proof}

\subsection{A topological persistence lemma}

\begin{lemma}[Persistence of strict local maxima]\label{lem:persistence}
Let compact metric spaces $Y_j$ and $Y$ be realized as compact subsets of a common
compact metric space $Z$.  Assume that there are numbers $\rho_j\downarrow0$ such
that
\begin{equation}\label{eq:persistence-Hausdorff}
d_H^Z(Y_j,Y)\le \rho_j,
\end{equation}
where $d_H^Z$ denotes Hausdorff distance in the ambient space $Z$.  Assume also
that there are continuous functions $v_j:Y_j\to\R$, $v:Y\to\R$ such that
\begin{equation}\label{eq:persistence-alpha}
\alpha_j:=
\sup\bigl\{|v_j(y_j)-v(y)|:\ y_j\in Y_j,\ y\in Y,\ d_Z(y_j,y)\le\rho_j\bigr\}
\longrightarrow0.
\end{equation}
If $v$ has $m$ distinct isolated strict local maxima, then, for all sufficiently
large $j$, there are pairwise disjoint relatively open sets
$W_{1,j},\ldots,W_{m,j}\subset Y_j$ such that
\begin{equation}\label{eq:persistence-boundary-gap-conclusion}
\max_{\overline W_{i,j}}v_j>
\sup_{\partial_{Y_j}W_{i,j}}v_j
\qquad(i=1,\ldots,m),
\end{equation}
where closures and boundaries are taken relative to $Y_j$ and
$\sup\emptyset=-\infty$.  In particular, for all sufficiently large $j$, $v_j$
has at least $m$ distinct local maxima.
\end{lemma}

\begin{proof}
Let $p_1,\ldots,p_m$ be distinct isolated strict local maxima of $v$.  For each
$i$, choose a relative open neighborhood $U_i\subset Y$ of $p_i$ such that
$v(y)<v(p_i)$ for every $y\in U_i\setminus\{p_i\}$.  Choose $r>0$ so small
that the closed $Z$-balls $\overline B_Z(p_i,4r)$ are pairwise disjoint and
$$
Y\cap\overline B_Z(p_i,4r)\subset U_i
\qquad(i=1,\ldots,m).
$$
For each $i$, set
\begin{equation}\label{eq:persistence-annulus}
A_i:=Y\cap\bigl(\overline B_Z(p_i,4r)\setminus B_Z(p_i,r)\bigr).
\end{equation}
If $A_i\ne\emptyset$, then $A_i$ does not contain $p_i$, and strict local
maximality gives
\begin{equation}\label{eq:persistence-annulus-gap}
\sup_{A_i}v\le v(p_i)-4\eta_i
\end{equation}
for some $\eta_i>0$.  If $A_i=\emptyset$, choose any $\eta_i>0$; in that case the
annular upper estimate below is vacuous for all large $j$.

For large $j$, \eqref{eq:persistence-Hausdorff} gives $a_{i,j}\in Y_j$ with
$d_Z(a_{i,j},p_i)\le\rho_j$.  Hence \eqref{eq:persistence-alpha} gives
\begin{equation}\label{eq:persistence-lower-test-point}
v_j(a_{i,j})\ge v(p_i)-\alpha_j>v(p_i)-\eta_i
\end{equation}
for all sufficiently large $j$.

Consider a point
$$
y_j\in Y_j\cap\bigl(\overline B_Z(p_i,3r)\setminus B_Z(p_i,2r)\bigr).
$$
For large $j$, there exists $y\in Y$ with $d_Z(y_j,y)\le\rho_j<r$.  Then
$y\in A_i$ by \eqref{eq:persistence-annulus}.  Hence, if $A_i=\emptyset$, no such
$y_j$ exists for large $j$.  If $A_i\ne\emptyset$, then using
\eqref{eq:persistence-annulus-gap} and \eqref{eq:persistence-alpha} gives
\begin{equation}\label{eq:persistence-upper-annulus}
v_j(y_j)\le v(y)+\alpha_j\le v(p_i)-4\eta_i+\alpha_j
< v(p_i)-2\eta_i
\end{equation}
for all sufficiently large $j$.

Set
$$
W_{i,j}:=Y_j\cap B_Z(p_i,3r).
$$
These sets are relatively open in $Y_j$ and pairwise disjoint for all large $j$.
The point $a_{i,j}$ belongs to $W_{i,j}$ for all large $j$, and
\eqref{eq:persistence-lower-test-point} gives
\begin{equation}\label{eq:persistence-W-lower}
\max_{\overline W_{i,j}}v_j\ge v_j(a_{i,j})>v(p_i)-\eta_i .
\end{equation}
The relative boundary $\partial_{Y_j}W_{i,j}$ is contained in
$Y_j\cap(\overline B_Z(p_i,3r)\setminus B_Z(p_i,2r))$; in particular the
annular estimate above gives
\begin{equation}\label{eq:persistence-W-boundary-upper}
\sup_{\partial_{Y_j}W_{i,j}}v_j\le v(p_i)-2\eta_i
\end{equation}
whenever the boundary is non-empty.  If the boundary is empty, the same statement
holds with the convention $\sup\emptyset=-\infty$.  Combining
\eqref{eq:persistence-W-lower} and
\eqref{eq:persistence-W-boundary-upper} proves
\eqref{eq:persistence-boundary-gap-conclusion}.  Any point where $v_j$ attains
its maximum on $\overline W_{i,j}$ must lie in $W_{i,j}$ and is therefore a local
maximum of $v_j$ on $Y_j$.  Since the sets $W_{i,j}$ are pairwise disjoint, these
local maxima are distinct.
\end{proof}

\subsection{A generic perturbation lemma}

\begin{proposition}[Non-degenerate perturbation]\label{prop:nondegenerate-perturbation}
Let $(M,g)$ be a smooth closed Riemannian surface with positive Gaussian
curvature.  Assume that $\lambda_1(g)$ is simple, and let $u$ be an
$L^2$-normalized $\lambda_1(g)$-eigenfunction.  Suppose that there are pairwise
disjoint open sets $W_1,\ldots,W_m\subset M$, each with non-empty boundary, such
that
\begin{equation}\label{eq:nondegenerate-perturbation-gap}
\max_{\overline W_i}u>
\sup_{\partial W_i}u
\qquad(i=1,\ldots,m),
\end{equation}
where closures and boundaries are taken in $M$.  Then, in every $C^\infty$
neighborhood of $g$, there exists a smooth metric $\widetilde g$ such that
$K_{\widetilde g}>0$, $\lambda_1(\widetilde g)$ is simple, and, after choosing
its sign, an $L^2$-normalized $\lambda_1(\widetilde g)$-eigenfunction has at
least $m$ distinct non-degenerate local maxima.

If, in addition, $M$ is diffeomorphic to $\mathbb S^2$, then $(M,\widetilde g)$
admits a smooth isometric embedding in $\R^3$ as a strictly convex surface.
\end{proposition}

\begin{proof}
Let
$$
\sigma:=\min_{1\le i\le m}
\left(\max_{\overline W_i}u-\sup_{\partial W_i}u\right)>0.
$$
Since $\lambda_1(g)$ is simple and
$$
0=\lambda_0(g)<\lambda_1(g)<\lambda_2(g),
$$
standard elliptic perturbation theory for a simple eigenvalue implies that, after
choosing signs, the $L^2$-normalized first positive eigenfunction depends
continuously on the metric in the $C^0$ topology under sufficiently small
$C^\infty$ perturbations of $g$.  Thus there is a $C^\infty$ neighborhood
$\mathcal U_0$ of $g$ such that, for every $h\in\mathcal U_0$, the normalized
$\lambda_1(h)$-eigenfunction $u_h$ can be chosen so that
\begin{equation}\label{eq:eigenfunction-C0-under-perturbation}
\|u_h-u\|_{C^0(M)}<\frac{\sigma}{4}.
\end{equation}
Shrinking $\mathcal U_0$ if necessary, we may also assume that
$K_h>0$ and that $\lambda_1(h)$ remains simple for all $h\in\mathcal U_0$; the
first assertion follows from the $C^2$-openness of positive Gaussian curvature,
and the second from spectral-gap continuity.

For every $h\in\mathcal U_0$, estimate
\eqref{eq:eigenfunction-C0-under-perturbation} gives
$$
\max_{\overline W_i}u_h
\ge \max_{\overline W_i}u-\frac{\sigma}{4}
> \sup_{\partial W_i}u+\frac{3\sigma}{4}
\ge \sup_{\partial W_i}u_h+\frac{\sigma}{2}.
$$
Hence $u_h$ has at least one local maximum in each $W_i$.

Let $\mathcal U$ be any prescribed $C^\infty$ neighborhood of $g$.  By
Uhlenbeck's generic theorem \cite{Uhlenbeck1976}, the set of smooth metrics for
which all non-zero eigenvalues are simple and all corresponding eigenfunctions
are Morse is residual and dense.  Therefore we may choose
$\widetilde g\in\mathcal U\cap\mathcal U_0$ with this generic property.  The
preceding boundary-gap argument applied to $h=\widetilde g$ gives at least one
local maximum of the normalized $\lambda_1(\widetilde g)$-eigenfunction in each
$W_i$.  Since this eigenfunction is Morse, all its critical points are
non-degenerate; in particular these local maxima are non-degenerate.  The sets
$W_i$ are pairwise disjoint, so the local maxima obtained in this way are
distinct.

If $M$ is diffeomorphic to $\mathbb S^2$, the smooth solution of Weyl's isometric
embedding problem for positively curved metrics on the two-sphere, due to
Nirenberg and Pogorelov \cite{Nirenberg1953,Pogorelov1952}, gives a smooth
isometric embedding of $(M,\widetilde g)$ into $\R^3$ as a strictly convex
surface.
\end{proof}

\section{Proof of the main theorem}\label{sec:proof-main}

\begin{proof}[Proof of Theorem~\ref{thm:main}]
Fix $m\ge2$.  Choose the Miyamoto polygon $\Om$ and eigenfunction $\phi$ from
Theorem~\ref{thm:miyamoto}.  Let $X=D\Om$ and let $\Phi$ be the normalized even
double from \eqref{eq:even-double-eigenfunction}.  By
Proposition~\ref{prop:limit-spectrum}, $\Phi$ is the unique normalized
$\lambda_1(X)$-eigenfunction up to sign and has at least $m$ isolated strict
local maxima.

Let $M_j$ be the smooth positively curved pancakes from
Proposition~\ref{prop:pancake-approx}.  By
Proposition~\ref{prop:eigenfunction-convergence}, for all sufficiently large $j$
the eigenvalue $\lambda_1(M_j)$ is simple; after passing to a subsequence and
choosing signs, normalized $\lambda_1(M_j)$-eigenfunctions $u_j$ converge
uniformly to $\Phi$ along measured Gromov--Hausdorff approximations.

We now verify the hypotheses of Lemma~\ref{lem:persistence}.  By the concrete
form of Honda convergence stated in Theorem~\ref{thm:honda}, after passing to
the same subsequence we have measured Gromov--Hausdorff approximation maps
$\mathcal F_j:M_j\to X$ and numbers $\gamma_j,\beta_j\downarrow0$ such that
$\mathcal F_j$ is a $\gamma_j$-Gromov--Hausdorff approximation and
\begin{equation}\label{eq:main-uniform-Honda-convergence}
\sup_{q\in M_j}|u_j(q)-\Phi(\mathcal F_j(q))|\le \beta_j.
\end{equation}
By the standard realization theorem for Gromov--Hausdorff convergence, realize
$M_j$ and $X$ isometrically in a common compact metric space $Z$ so that
\begin{equation}\label{eq:main-approximation-maps}
d_Z(q,\mathcal F_j(q))\le \gamma_j
\qquad(q\in M_j).
\end{equation}
Since $\mathcal F_j(M_j)$ is $\gamma_j$-dense in $X$, this realization also gives
$d_H^Z(M_j,X)\le 2\gamma_j$.  Set
$$
\rho_j:=2\gamma_j,
$$
so that $\rho_j\downarrow0$ and
\begin{equation}\label{eq:main-Hausdorff-realization}
d_H^Z(M_j,X)\le \rho_j,
\qquad
\gamma_j\le \rho_j.
\end{equation}
Let $\omega_\Phi$ be a modulus of continuity of $\Phi$ on the compact space $X$.
If $q\in M_j$, $y\in X$, and $d_Z(q,y)\le\rho_j$, then, by
\eqref{eq:main-approximation-maps},
$$
d_X(\mathcal F_j(q),y)=d_Z(\mathcal F_j(q),y)
\le d_Z(\mathcal F_j(q),q)+d_Z(q,y)
\le \gamma_j+\rho_j
\le 2\rho_j.
$$
Here the equality holds because $\mathcal F_j(q)$ and $y$ both belong to the
isometrically embedded copy of $X$ inside $Z$.
Therefore, by \eqref{eq:main-uniform-Honda-convergence},
\begin{equation}\label{eq:main-persistence-alpha-bound}
|u_j(q)-\Phi(y)|
\le |u_j(q)-\Phi(\mathcal F_j(q))|+|\Phi(\mathcal F_j(q))-\Phi(y)|
\le \beta_j+\omega_\Phi(2\rho_j)\to0.
\end{equation}
Equations \eqref{eq:main-Hausdorff-realization} and
\eqref{eq:main-persistence-alpha-bound} are precisely the hypotheses
\eqref{eq:persistence-Hausdorff} and \eqref{eq:persistence-alpha} of
Lemma~\ref{lem:persistence}, with $Y_j=M_j$, $Y=X$, $v_j=u_j$, and $v=\Phi$.
Thus, for all sufficiently large $j$, there are pairwise disjoint open sets
$W_{1,j},\ldots,W_{m,j}\subset M_j$ such that
\begin{equation}\label{eq:main-boundary-gaps-on-pancake}
\max_{\overline W_{i,j}}u_j>
\sup_{\partial W_{i,j}}u_j
\qquad(i=1,\ldots,m).
\end{equation}
In particular, $u_j$ has at least $m$ distinct local maxima. Since $m\ge2$ and $M_j$ is connected, each $W_{i,j}$ is a proper non-empty
open subset of $M_j$. Hence $\partial W_{i,j}\ne\emptyset$ for every $i$.

Choose such a sufficiently large $j$.  Let $g_j$ be the induced metric on the
smooth embedded pancake $M_j$.  Applying
Proposition~\ref{prop:nondegenerate-perturbation} to
$(M_j,g_j)$, the eigenfunction $u_j$, and the sets
$W_{1,j},\ldots,W_{m,j}$, we obtain a smooth metric $\widetilde g_j$ on the
underlying two-sphere, arbitrarily close to $g_j$, such that
$K_{\widetilde g_j}>0$, $\lambda_1(\widetilde g_j)$ is simple, and a normalized
$\lambda_1(\widetilde g_j)$-eigenfunction $\widetilde u_j$, after choosing its
sign, has at least $m$ distinct non-degenerate local maxima.

By the last assertion of Proposition~\ref{prop:nondegenerate-perturbation},
there is a smooth isometric embedding
$$
F_j:(M_j,\widetilde g_j)\longrightarrow \R^3
$$
as a strictly convex surface.  Set
$$
M_m:=F_j(M_j)\subset\R^3.
$$
The induced metric on $M_m$ is isometric to $(M_j,\widetilde g_j)$, hence has
positive Gaussian curvature and simple first positive eigenvalue.  Transporting
$\widetilde u_j$ by the isometry gives a normalized first nonzero eigenfunction
on $M_m$ with at least $m$ distinct non-degenerate local maxima.  Since the first
eigenspace is one-dimensional, the opposite sign has at least $m$ distinct
non-degenerate local minima.
\end{proof}

\subsection*{Availability of Data}
No datasets were generated or analyzed in this work.

\subsection*{Declarations of Conflict of Interest}
The author declares no conflict of interest.

\end{document}